\documentclass[11pt,reqno]{amsart}
\usepackage{etoolbox}
\makeatletter
\patchcmd{\@sect}
  {\@hangfrom{\hskip #3\relax\@svsec}{\interlinepenalty\@M #8\par}}
  {\ifnum#2=\@ne
     \centering\interlinepenalty\@M\hskip #3\relax\@svsec#8\par
   \else
     \@hangfrom{\hskip #3\relax\@svsec}{\interlinepenalty\@M #8\par}
   \fi}
  {}{\PackageError{manuscript}{Section heading patch failed}{}}
\makeatother
\usepackage[a4paper,textwidth=154mm,hcentering,top=28mm,bottom=28mm]{geometry}
\usepackage{amsmath,amssymb,amsthm,booktabs,hyperref}
\hypersetup{colorlinks=true,urlcolor=blue,linkcolor=blue,citecolor=blue,filecolor=blue,pdfborder={0 0 0}}
\numberwithin{equation}{section}
\newtheorem{theorem}{Theorem}[section]
\newtheorem{lemma}[theorem]{Lemma}
\newtheorem{proposition}[theorem]{Proposition}
\newtheorem{corollary}[theorem]{Corollary}
\theoremstyle{definition}

\theoremstyle{remark}

\theoremstyle{definition}
\newtheorem{problem}[theorem]{Problem}
\newtheorem{conjecture}[theorem]{Conjecture}
\newcommand{\Rpoly}[1]{\mathcal R_{#1}^{\vee}}
\DeclareMathOperator{\conv}{conv}
\DeclareMathOperator{\Ehr}{Ehr}
\DeclareMathOperator{\Vol}{Vol}
\title[The Athanasiadis--Chapoton ordinal-sum conjecture]{A proof of the Athanasiadis--Chapoton ordinal-sum conjecture}
\author{Congyi Luo}
\address{School of Data Science, Fudan University, Shanghai 200433, China}
\email{cyluo24@m.fudan.edu.cn}
\keywords{Preorder, ordinal sum, Ehrhart series, reflexive polytope, piecewise unimodular map}
\date{}
\begin{document}
\begin{abstract}
Given a finite preorder $\tau$, we study how its order relation determines the $h^*$-polynomial of the associated polar preorder polytope. Let $\Rpoly{\tau}$ be the convex hull of the negative standard basis vectors and the indicator vectors of nonempty order ideals, and write $h_\tau^*(t)=h^*(\Rpoly{\tau},t)$. Athanasiadis and Chapoton conjectured that this enumerative invariant takes ordinal sums to products:
\[
 h_{\tau_1\oplus\tau_2}^*(t)=h_{\tau_1}^*(t)\,h_{\tau_2}^*(t).
\]
We prove the conjecture for arbitrary finite preorders, including those with nontrivial equivalence classes. More precisely, we construct a positively homogeneous piecewise unimodular shear from the free sum $\Rpoly{\tau_1}\oplus\Rpoly{\tau_2}$ onto $\Rpoly{\tau_1\oplus\tau_2}$, giving lattice-point bijections at every nonnegative real dilation scale. The key is a formula for the Minkowski functional in terms of the least nonnegative order-reversing majorant: after the shear, the functional becomes the sum of the two factor functionals. Together with the Ehrhart product formula for free sums, this proves the conjecture and the corresponding product formula for normalized volume.
\end{abstract}
\maketitle
\section{Introduction and main results}
Athanasiadis and Chapoton~\cite[Conjecture~6.3]{AC} conjectured that the ordinal sum of finite preorders corresponds to multiplication of the $h^*$-polynomials of their polar preorder polytopes. We prove this conjecture and give an explicit piecewise unimodular map realizing the factorization.

\subsection{The ordinal-sum conjecture and the product formula}
We first describe the objects in the conjecture. Let $\tau$ be a preorder on a finite set $E$, that is, a reflexive and transitive relation $\le_\tau$, not necessarily antisymmetric. A subset $I\subseteq E$ is an \emph{order ideal} if $j\in I$ and $i\le_\tau j$ imply $i\in I$. Write $\mathcal J(\tau)$ for the set of all order ideals, $e_i$ for the standard basis vectors of $\mathbb R^E$, and $\mathbf1_I=\sum_{i\in I}e_i$ for the indicator vector of $I$. Following the notation of~\cite[Section~6]{AC}, the polar preorder polytope is
\begin{equation}\label{polytope}
 \Rpoly{\tau}
 =\conv\bigl(\{-e_i:i\in E\}\cup
       \{\mathbf1_I:\varnothing\ne I\in\mathcal J(\tau)\}\bigr).
\end{equation}
Here $\conv$ denotes the convex hull and $\vee$ denotes polar duality. This polytope has dimension $|E|$ and contains the origin in its interior; a direct proof is given in Section~\ref{sec:gauge}. If $E=\varnothing$, we set $\Rpoly{\varnothing}=\{0\}\subset\mathbb R^\varnothing$.

Put $d=|E|$. For an integer $m\ge0$, write
$L_\tau(m)=\#(m\Rpoly{\tau}\cap\mathbb Z^E)$. Ehrhart's theorem gives
\begin{equation}\label{ehrdefinition}
 \Ehr_\tau(t):=\sum_{m\ge0}L_\tau(m)t^m
       =\frac{h_\tau^*(t)}{(1-t)^{d+1}},
 \qquad h_\tau^*(t):=h^*(\Rpoly{\tau},t),
\end{equation}
where $h_\tau^*(t)$ is a polynomial of degree at most $d$ and $t$ is a formal variable; see~\cite[Sections~3.4 and~3.5]{BR}. Thus $h_\tau^*(t)$, together with the size of the ground set, determines the lattice-point counts in all integral dilates. In particular, $h_\varnothing^*(t)=1$.

Suppose that $\tau_1,\tau_2$ have disjoint ground sets. Their \emph{ordinal sum} $\tau_1\oplus\tau_2$ retains the relations within each ground set and places every element of the first ground set below every element of the second. The conjecture of Athanasiadis and Chapoton is precisely the following product formula.

\begin{theorem}\label{main}
For finite preorders $\tau_1,\tau_2$ on disjoint ground sets,
\begin{equation}\label{hproduct}
 h^*(\Rpoly{\tau_1\oplus\tau_2},t)
       =h^*(\Rpoly{\tau_1},t)\,h^*(\Rpoly{\tau_2},t).
\end{equation}
\end{theorem}
This answers~\cite[Conjecture~6.3]{AC} affirmatively. If $\tau=\tau_1\oplus\cdots\oplus\tau_k$, iteration gives
\[
 h_\tau^*(t)=\prod_{j=1}^k h_{\tau_j}^*(t).
\]
Nontrivial equivalence classes need not be split or collapsed: all lattice points are counted in the integer lattice indexed by the original ground set.

\subsection{The geometric theorem and the proof strategy}
The product formula arises from an explicit lattice-point bijection between two polytopes. For polytopes $P\subset\mathbb R^A$ and $Q\subset\mathbb R^B$ containing the origin, with $A\cap B=\varnothing$, their \emph{free sum} is
\[
 P\oplus Q=\conv\bigl((P\times\{0\})\cup(\{0\}\times Q)\bigr).
\]
We use the usual symbol $\oplus$ for both ordinal sums and free sums; the operands distinguish the two operations. For a real vector $z=(z_i)$, write
$z_+=(\max\{z_i,0\})_i$, and use the usual norms
$\|z\|_1=\sum_i|z_i|$ and $\|z\|_\infty=\max_i|z_i|$. Both norms of the empty vector are defined to be zero.

A map is \emph{piecewise unimodular} if the ambient space is covered by finitely many polyhedral cones on each of which the map is given by an integer matrix of determinant $\pm1$. The following map and its inverse both have this property.

\begin{theorem}\label{geometric}
Let $\tau_1,\tau_2$ be finite preorders on nonempty disjoint ground sets $A,B$. Define
$\Phi:\mathbb R^A\times\mathbb R^B\to\mathbb R^A\times\mathbb R^B$ by
\begin{equation}\label{shearmap}
 \Phi(x,y)=\bigl(x+\|y_+\|_\infty\mathbf1_A,\,y\bigr)
\end{equation}
Then $\Phi$ is a positively homogeneous piecewise unimodular homeomorphism, and
\begin{equation}\label{dilationmap}
 \Phi\bigl(m(\Rpoly{\tau_1}\oplus\Rpoly{\tau_2})\bigr)
       =m\Rpoly{\tau_1\oplus\tau_2}
       \qquad(m\ge0).
\end{equation}
For every real $m\ge0$, $\Phi$ restricts to a bijection between the lattice-point sets on the two sides of~\eqref{dilationmap}.
\end{theorem}
Positive homogeneity means that $\Phi(cz)=c\Phi(z)$ for $c\ge0$. Thus the geometric correspondence at unit scale determines the correspondence at every scale. The inverse shear has the equally simple form
\[
 \Phi^{-1}(u,v)=\bigl(u-\|v_+\|_\infty\mathbf1_A,\,v\bigr).
\]
The main point is to prove~\eqref{dilationmap}: a piecewise linear map need not preserve convex hulls, so comparing vertices alone does not determine the image of the whole polytope.

We address this difficulty using the Minkowski functional. For a nonempty finite preorder $\tau$, let
\[
 \gamma_\tau(z)=\inf\{r\ge0:z\in r\Rpoly{\tau}\}.
\]
This is the smallest dilation factor for which the dilate contains $z$. Section~\ref{sec:gauge} gives an explicit formula for $\gamma_\tau$, from which Section~\ref{sec:shear} proves
\begin{equation}\label{introidentity}
 \gamma_{\tau_1\oplus\tau_2}\bigl(\Phi(x,y)\bigr)
       =\gamma_{\tau_1}(x)+\gamma_{\tau_2}(y).
\end{equation}
The right-hand side is the Minkowski functional of the free sum, so this identity proves Theorem~\ref{geometric}. The functionals take integer values on lattice points. Convolution over these integer scales in Section~\ref{sec:ehrhart} then yields Theorem~\ref{main}.

\subsection{Relation to earlier work}
Multiplicativity for free sums is already well understood. A lattice polytope $P$ containing the origin in its interior is called
\emph{reflexive} if its polar dual $P^\vee=\{u:\langle u,x\rangle\le1\text{ for all }x\in P\}$
is also a lattice polytope. Braun~\cite[Theorem~1]{BRAUN} proved that if $P$ is reflexive and $Q$ is a lattice polytope containing the origin in its interior, then
\[
 \Ehr_{P\oplus Q}(t)=(1-t)\Ehr_P(t)\Ehr_Q(t).
\]
Here, for a full-dimensional lattice polytope $P\subset\mathbb R^d$, $\Ehr_P(t)=\sum_{m\ge0}\#(mP\cap\mathbb Z^d)t^m$. Beck, Jayawant, and McAllister~\cite[Theorems~1.3 and~1.4]{BJM} established necessary and sufficient conditions for this product formula, as well as a multivariate version. For two full-dimensional lattice factors containing the origin in their interiors, the product formula holds if and only if at least one factor is reflexive.

Bruns~\cite[Sections~1 and~3]{BRUNS} also gave an algebraic interpretation: identifying the degree-one elements representing the origins in the two factors accounts for the factor $1-t$ in the series. Section~\ref{sec:ehrhart} derives this factor directly by grouping lattice points according to their dilation scales.

Athanasiadis and Chapoton proved that $\Rpoly{\tau}$ is reflexive~\cite[Proposition~6.1]{AC}, so the counting formula for the free sum applies. The missing geometric link is that the polytope of an ordinal sum is generally not affinely equivalent to the free sum. Theorem~\ref{geometric} supplies this link by an explicit piecewise unimodular homeomorphism.

Piecewise maps have long played a role in lattice-point enumeration for poset polytopes. Stanley's transfer map~\cite[Theorem~3.2]{ST86} relates the order polytope and the chain polytope of a poset and preserves the lattice-point counts in every integral dilate. The combinatorial mutations studied by Akhtar, Coates, Galkin, and Kasprzyk~\cite[Section~3]{ACGK} likewise preserve Ehrhart series through piecewise unimodular maps on the dual space. Higashitani~\cite[Definition~3.1 and Theorem~4.1]{HIG} applied this construction to order and chain polytopes. Our shear is a particular instance of such a dual mutation map; its parameters are given in Section~\ref{subsec:mutation}.

Reflexivity also implies symmetry of $h_\tau^*(t)$\cite[Theorem~2.1]{HIBI}. This symmetry neither explains the factorization under ordinal sums nor asserts invariance under reversing the preorder. Section~\ref{sec:discussion} states the latter conjecture separately from a stronger geometric question and explains which cases reduce to smaller factors. The appendix gives an explicit calculation when the equivalence classes are totally ordered.

\section{The Minkowski functional of a polar preorder polytope}\label{sec:gauge}
We determine the smallest dilation factor for $\Rpoly{\tau}$. Starting from a vector, we raise its coordinates to obtain a nonnegative order-reversing vector and then use negative standard basis vectors to recover the coordinate differences. The least such majorant gives the formula for the functional. This identifies the whole polytope and explains why lattice points have integer dilation scales.

Throughout this section, $E\ne\varnothing$, and inequalities between vectors are coordinatewise. We first ensure that the smallest dilation factor is defined everywhere.
\begin{lemma}\label{interior}
The polytope $\Rpoly{\tau}$ is full-dimensional and contains the origin in its interior.
\end{lemma}
\begin{proof}
Put $d=|E|$. Since $E$ is an order ideal, $\Rpoly{\tau}$ contains the simplex
\[
 S=\conv\bigl(\{-e_i:i\in E\}\cup\{\mathbf1_E\}\bigr).
\]
To prove full dimensionality, suppose that $\sum_i c_i(-e_i-\mathbf1_E)=0$ and put $c=\sum_i c_i$. The $j$th coordinate gives $c_j=-c$; summing yields $c=-dc$. Hence every $c_j=0$, and the vertices are affinely independent. Moreover,
\[
 0=\frac{1}{d+1}\left(\mathbf1_E+\sum_{i\in E}(-e_i)\right),
\]
all barycentric coordinates of the origin are positive. Thus the origin lies in the interior of $S$, and hence in the interior of $\Rpoly{\tau}$.
\end{proof}

\subsection{The least nonnegative order-reversing majorant}\label{sec:envelope}
Indicator vectors of order ideals are order-reversing. Their nonnegative combinations are therefore nonnegative order-reversing vectors. We show that the least such vector dominating a given vector is obtained by taking maxima over upper sets, and that its level sets express it as a combination of indicator vectors of order ideals.

A vector $v\in\mathbb R^E$ is \emph{order-reversing} if $i\le_\tau j$ implies $v_i\ge v_j$. For $z\in\mathbb R^E$, define its majorant $\widehat z$ by
\begin{equation}\label{envelopedef}
 \widehat z_i=\max_{j:\,i\le_\tau j}z_j^+,
 \qquad z_j^+=\max\{z_j,0\}.
\end{equation}
Reflexivity of the preorder ensures that every maximum is taken over a nonempty finite set. The notation $\widehat z$ always refers to the preorder under consideration.

\begin{lemma}\label{envelope}
The vector $\widehat z$ is the coordinatewise least among all nonnegative order-reversing vectors $v$ satisfying $v\ge z$, and
\begin{equation}\label{maxenvelope}
 \|\widehat z\|_\infty=\|z_+\|_\infty.
\end{equation}
\end{lemma}
\begin{proof}
By definition, $\widehat z_i\ge z_i^+\ge z_i$ and $\widehat z_i\ge0$. If $i\le_\tau j$, transitivity gives
$\{k:j\le_\tau k\}\subseteq\{k:i\le_\tau k\}$, so $\widehat z_i\ge\widehat z_j$.
Conversely, if $v$ is nonnegative and order-reversing with $v\ge z$, then for every $j\ge_\tau i$ we have $v_i\ge v_j\ge z_j$ and $v_i\ge0$. Hence $v_i\ge\widehat z_i$, proving minimality.

Every $\widehat z_i$ is at most $\|z_+\|_\infty$. Choose $j$ such that $z_j^+=\|z_+\|_\infty$. Reflexivity gives $\widehat z_j\ge z_j^+$, proving~\eqref{maxenvelope}.
\end{proof}

If $i\sim_\tau j$, meaning that $i\le_\tau j$ and $j\le_\tau i$, then $\widehat z_i=\widehat z_j$. Here $\sim_\tau$ is the equivalence relation associated with the preorder; the input coordinates $z_i,z_j$ may still differ. The relation between two equivalence classes is independent of their representatives, and two classes lying below each other must coincide. Thus the equivalence classes form a poset, called the \emph{quotient poset}.

\begin{lemma}\label{levels}
Let $v$ be a nonnegative order-reversing vector whose distinct positive coordinate values are $a_1<\cdots<a_s$. Put $a_0=0$ and $I_k=\{i:v_i\ge a_k\}$. Then each $I_k$ is a nonempty order ideal, and
\begin{equation}\label{leveldecomposition}
 v=\sum_{k=1}^s(a_k-a_{k-1})\mathbf1_{I_k},
 \qquad \sum_{k=1}^s(a_k-a_{k-1})=\|v\|_\infty.
\end{equation}
When $v=0$, both sums are empty.
\end{lemma}
\begin{proof}
The assertion holds for $v=0$, so suppose that $v\ne0$. Since $a_k$ is a coordinate value, $I_k$ is nonempty. If $j\in I_k$ and $i\le_\tau j$, then $v_i\ge v_j\ge a_k$, so $i\in I_k$.

If $v_i=0$, the $i$th coordinate of the right-hand side is zero. If $v_i=a_l$, then $i\in I_k$ if and only if $k\le l$, so its $i$th coordinate is
$\sum_{k=1}^l(a_k-a_{k-1})=a_l$. This proves the vector identity. The sum of all coefficients is $a_s=\|v\|_\infty$.
\end{proof}
Applying this decomposition to $\widehat z$ expresses it in terms of indicator vectors of order ideals with total coefficient $\|z_+\|_\infty$. Recovering $z$ from the nonnegative coordinate difference requires negative standard basis vectors with total coefficient $\|\widehat z-z\|_1$.

\subsection{The functional and integer dilation scales}
The sum of these two groups of coefficients is the smallest dilation factor. Its expression involves only maxima and two standard vector norms, allowing us to compare the polytopes of different preorders without computing their facet descriptions.

\begin{proposition}\label{gauge}
For every nonempty finite preorder $\tau$ and every $z\in\mathbb R^E$,
\begin{equation}\label{gdef}
 \gamma_\tau(z)=\|z_+\|_\infty+\|\widehat z-z\|_1.
\end{equation}
In particular, for every real $m\ge0$,
\begin{equation}\label{sublevel}
 m\Rpoly{\tau}=\{z:\gamma_\tau(z)\le m\},
 \qquad \gamma_\tau(\mathbb Z^E)\subseteq\mathbb Z_{\ge0}.
\end{equation}
\end{proposition}
\begin{proof}
Denote the right-hand side of~\eqref{gdef} by $\rho(z)$. Since $\widehat z\ge z$, the second term is $\sum_i(\widehat z_i-z_i)$, with each summand nonnegative. We first prove that $\Rpoly{\tau}=\{z:\rho(z)\le1\}$.

If $z\in\Rpoly{\tau}$, the convex-hull definition gives an expression
\[
 z=p-\beta,\qquad
 p=\sum_{\varnothing\ne I\in\mathcal J(\tau)}\alpha_I\mathbf1_I,
 \qquad \alpha_I,\beta_i\ge0,\quad
 \sum_I\alpha_I+\sum_i\beta_i=1,
\]
where $\beta=(\beta_i)_{i\in E}$. Put $\lambda=\sum_I\alpha_I$. The vector $p$ is nonnegative, order-reversing, and dominates $z$, so Lemma~\ref{envelope} gives $\widehat z\le p$. Also, $z_i\le p_i\le\lambda$ and $\lambda\ge0$, whence $\|z_+\|_\infty\le\lambda$. Therefore
\[
 \rho(z)\le\lambda+\sum_i(p_i-z_i)
          =\lambda+\sum_i\beta_i=1.
\]

Conversely, if $\rho(z)\le1$, Lemma~\ref{levels} and~\eqref{maxenvelope} give an expression
\[
 \widehat z=\sum_{k=1}^s\alpha_k\mathbf1_{I_k},
 \qquad \alpha_k\ge0,\qquad
 \sum_k\alpha_k=\|z_+\|_\infty.
\]
Consequently,
\[
 z=\sum_{k=1}^s\alpha_k\mathbf1_{I_k}
       +\sum_{i\in E}(\widehat z_i-z_i)(-e_i).
\]
All coefficients are nonnegative and sum to $\rho(z)\le1$. Adding the origin with coefficient $1-\rho(z)$ gives a convex combination in $\Rpoly{\tau}$, so $z\in\Rpoly{\tau}$. If $\widehat z=0$, the first sum is empty and the argument is unchanged.

For $c\ge0$, the definition gives $\widehat{cz}=c\widehat z$, hence $\rho(cz)=c\rho(z)$. Thus, for $m>0$,
\[
 z\in m\Rpoly{\tau}
 \ \Longleftrightarrow\ z/m\in\Rpoly{\tau}
 \ \Longleftrightarrow\ \rho(z)\le m.
\]
If $\rho(z)=0$, then $z_+=0$ and $\widehat z=z$. The first equality implies $z\le0$, while the second, together with $\widehat z\ge0$, gives $z=0$. The same sublevel-set identity therefore holds when $m=0$.

It follows that the set of scales $r$ for which $z\in r\Rpoly{\tau}$ is exactly $[\rho(z),\infty)$. Taking the infimum yields $\gamma_\tau(z)=\rho(z)$ and establishes~\eqref{sublevel}. Finally, for an integer input, $\widehat z$ is an integer vector and both terms in~\eqref{gdef} are nonnegative integers.
\end{proof}

Integer dilation scales are a manifestation of reflexivity~\cite[Lemma~1]{BRAUN}. The stronger information needed here is the explicit formula~\eqref{gdef}: it separates the scale into the maximum positive coordinate and the coordinate differences from the majorant. The ordinal-sum shear acts precisely on these two parts.

\section{A shear from the free sum to the ordinal sum}\label{sec:shear}
In an ordinal sum, the upper sets in the first factor share a common part: the entire ground set of the second factor. Translating the first coordinate block separates this common part from the majorant, making the functional equal to the sum of the two factor functionals. This identity gives the polytope correspondence in Theorem~\ref{geometric}.

Throughout this section, both ground sets $A,B$ are nonempty. For $x\in\mathbb R^A$ and $y\in\mathbb R^B$, let $\widehat x,\widehat y$ be their majorants with respect to $\tau_1,\tau_2$, respectively. For $z\in\mathbb R^{A\cup B}$, the notation $\widehat z$ refers to the ordinal sum.

\subsection{Additivity of the functional}
In the ordinal sum, the majorant of every $A$-coordinate takes account of the maximum positive coordinate in $B$. Translating all $A$-coordinates by the same amount preserves the differences between coordinates and their majorants, while the remaining maximum term splits into the sum of the maxima for the two factors.

\begin{lemma}\label{gadd}
For all $x\in\mathbb R^A$ and $y\in\mathbb R^B$,
\begin{equation}\label{addidentity}
 \gamma_{\tau_1\oplus\tau_2}\bigl(\Phi(x,y)\bigr)
       =\gamma_{\tau_1}(x)+\gamma_{\tau_2}(y).
\end{equation}
\end{lemma}
\begin{proof}
Put $b=\|y_+\|_\infty$ and $z=(x+b\mathbf1_A,y)$. For $i\in A$, the elements above $i$ in the ordinal sum are precisely the elements above $i$ in the original factor together with all of $B$. Hence
\[
 \widehat z_i
 =\max\bigl\{b,\ b+\max_{a:\,i\le_{\tau_1}a}x_a\bigr\}
 =b+\widehat x_i.
\]
For $j\in B$, the elements above $j$ remain entirely in $B$, so $\widehat z_j=\widehat y_j$. Similarly,
\[
 \|z_+\|_\infty
 =\max\bigl\{b,\ b+\max_{i\in A}x_i\bigr\}
 =b+\|x_+\|_\infty.
\]
Thus $\widehat z-z=(\widehat x-x,\widehat y-y)$. Substitution into Proposition~\ref{gauge} gives
\begin{align*}
 \gamma_{\tau_1\oplus\tau_2}(z)
 &=b+\|x_+\|_\infty
       +\|\widehat x-x\|_1+\|\widehat y-y\|_1\\
 &=\gamma_{\tau_1}(x)+\gamma_{\tau_2}(y),
\end{align*}
where the last equality uses $b=\|y_+\|_\infty$.
\end{proof}

\subsection{Unimodular pieces and the image of the polytope}
To obtain a lattice-point bijection from the functional identity, we describe the inverse shear and its linear pieces. The pieces depend only on the second coordinate block, so the map and its inverse use the same cones.

\begin{lemma}\label{pl}
The map $\Phi$ is a positively homogeneous piecewise unimodular homeomorphism and maps $\mathbb Z^{A\cup B}$ bijectively to itself.
\end{lemma}
\begin{proof}
The second coordinate block is unchanged, so
$\Phi^{-1}(u,v)=(u-\|v_+\|_\infty\mathbf1_A,v)$.
Direct composition gives the identity in both orders. The maximum of finitely many linear functions is continuous and commutes with nonnegative scaling. Hence $\Phi$ and its inverse are continuous and positively homogeneous.

Set
\[
 C_-=\{(x,y):y_j\le0\ (j\in B)\},\qquad
 C_j=\{(x,y):y_j\ge0,\ y_j\ge y_k\ (k\in B)\}.
\]
If every $y_j\le0$, the point lies in $C_-$. Otherwise, some $y_j$ attains a positive maximum and the point lies in $C_j$. These closed polyhedral cones therefore cover the whole space. On $C_-$, the map $\Phi$ is the identity. On $C_j$, its matrix and inverse matrix are
\[
 M_j=\begin{pmatrix}I_A&\mathbf1_Ae_j^{\mathsf T}\\0&I_B\end{pmatrix},
 \qquad
 M_j^{-1}=\begin{pmatrix}I_A&-\mathbf1_Ae_j^{\mathsf T}\\0&I_B\end{pmatrix}.
\]
Here $I_A,I_B$ are the corresponding identity matrices and $e_j\in\mathbb R^B$. Both matrices have integer entries and determinant one. On $C_j\cap C_k$ we have $y_j=y_k$, and on $C_-\cap C_j$ we have $y_j=0$, so the formulas agree on overlaps. Since $\Phi$ leaves $y$ unchanged, its inverse preserves these cones as well.

For an integer vector $y$, the number $\|y_+\|_\infty$ is an integer. Thus both $\Phi$ and $\Phi^{-1}$ preserve the integer lattice, proving the lattice-point bijection.
\end{proof}

The following description of the free sum shows that additivity of the functional is precisely the desired geometric correspondence.
\begin{lemma}\label{freesum}
For every real $m\ge0$,
\[
 m(\Rpoly{\tau_1}\oplus\Rpoly{\tau_2})
   =\{(x,y):\gamma_{\tau_1}(x)+\gamma_{\tau_2}(y)\le m\}.
\]
\end{lemma}
\begin{proof}
First take $m=1$. Grouping a convex combination in the free sum according to the two coordinate blocks, we may write
$(x,y)=\alpha(p,0)+\beta(0,q)$, where $p\in\Rpoly{\tau_1}$, $q\in\Rpoly{\tau_2}$, and $\alpha,\beta\ge0$, $\alpha+\beta=1$. If a weight is zero, choose the corresponding $p$ or $q$ to be the origin. Positive homogeneity gives
$\gamma_{\tau_1}(x)+\gamma_{\tau_2}(y)\le\alpha+\beta=1$.

Conversely, put $r=\gamma_{\tau_1}(x)$ and $s=\gamma_{\tau_2}(y)$, and suppose that $r+s\le1$. If $r>0$, take $p=x/r$, which belongs to $\Rpoly{\tau_1}$ by Proposition~\ref{gauge}. If $r=0$, then $x=0$ and we take $p=0$. Similarly choose $q\in\Rpoly{\tau_2}$ with $y=sq$. Then
\[
 (x,y)=r(p,0)+s(0,q)+(1-r-s)(0,0)
\]
belongs to the free sum. Positive homogeneity gives the result for $m>0$ by scaling. When $m=0$, the sum of the two functionals vanishes if and only if $x=y=0$.
\end{proof}

\begin{proof}[Proof of Theorem~\ref{geometric}]
Lemmas~\ref{gadd} and~\ref{freesum}, together with Proposition~\ref{gauge}, give
\[
 (x,y)\in m(\Rpoly{\tau_1}\oplus\Rpoly{\tau_2})
 \quad\Longleftrightarrow\quad
 \Phi(x,y)\in m\Rpoly{\tau_1\oplus\tau_2}.
\]
The bijections on the whole space and on the integer lattice from Lemma~\ref{pl} now imply all the assertions.
\end{proof}

\subsection{Relation to combinatorial mutation}\label{subsec:mutation}
The shear also fits the standard form of combinatorial mutation. This identifies the relation between the ordinal sum and the free sum with a specific class of piecewise transformations, beyond their Ehrhart equivalence, and provides a setting for the geometric question about opposite preorders.

In coordinates for the standard integer lattice and its dual, the dual map of a combinatorial mutation takes the form~\cite[Section~3]{ACGK}~\cite[Definition~3.1]{HIG}
\begin{equation}\label{mutationmap}
 \varphi_{w,F}(z)=z-\min_{v\in F}\langle z,v\rangle\,w,
\end{equation}
where $w$ is a primitive integer vector, meaning that it is not an integer multiple of another integer vector by a scalar of absolute value greater than one, and $F$ is a lattice polytope in the hyperplane $w^\perp$. If $Q$ is a full-dimensional rational polytope containing the origin and $\varphi_{w,F}(Q)$ is convex, then $Q$ and its image are related by a combinatorial mutation on the dual side~\cite[Proposition~3.4 and Definition~3.5]{HIG}. Our free sum is a full-dimensional lattice polytope containing the origin in its interior, so it lies in this class.

In our setting, take
\[
 w=(\mathbf1_A,0),\qquad
 F=\conv\bigl(\{0\}\cup\{(0,-e_j):j\in B\}\bigr).
\]
Since $A\ne\varnothing$, the vector $w$ is primitive, and clearly $F\subseteq w^\perp$. For $z=(x,y)$,
\[
 \min_{v\in F}\langle z,v\rangle
       =\min\bigl(\{0\}\cup\{-y_j:j\in B\}\bigr)
       =-\|y_+\|_\infty.
\]
Thus $\varphi_{w,F}=\Phi$, and Theorem~\ref{geometric} ensures that its image on the free sum is indeed convex. Here we use only the definition of the piecewise mutation map on the dual space; Ehrhart equivalence has already been proved directly above.

\section{The product formula and its consequences}\label{sec:ehrhart}
The shear reduces lattice-point enumeration to adding the dilation scales in the two factors. Since both scales are nonnegative integers on lattice points, we may count the new lattice points at each integer scale and then convolve these counts. This gives the product identity for Ehrhart series and completes the proof of the main theorem.

\subsection{Counting by integer dilation scales}
The integer-valued property of the functionals means that each lattice point first appears at an integer scale. Differences between the counts in consecutive dilates record these layers and turn Lemma~\ref{gadd} into a convolution. First suppose that $A,B$ are nonempty and write
\[
L_1(m)=L_{\tau_1}(m),\qquad L_2(m)=L_{\tau_2}(m),
\qquad L_1(-1)=L_2(-1)=0,
\]
For integers $r\ge0$, define $a_i(r)=L_i(r)-L_i(r-1)$.

\begin{lemma}\label{shell}
The number $a_i(r)$ counts the lattice points in the $i$th factor whose dilation scale is exactly $r$. If
$A_i(t)=\sum_{r\ge0}a_i(r)t^r$, then
\begin{equation}\label{shellseries}
A_i(t)=(1-t)\Ehr_{\tau_i}(t).
\end{equation}
\end{lemma}
\begin{proof}
For $r\ge1$, Proposition~\ref{gauge} and the integer-valued property give
\[
\{z\in\mathbb Z^{E_i}:\gamma_{\tau_i}(z)=r\}
 =\bigl(r\Rpoly{\tau_i}\cap\mathbb Z^{E_i}\bigr)
   \setminus\bigl((r-1)\Rpoly{\tau_i}\cap\mathbb Z^{E_i}\bigr),
\]
where $E_1=A,E_2=B$. The sublevel sets are nested, and the polytope is bounded, so these sets are finite. Taking cardinalities proves the assertion. When $r=0$, the only lattice point is the origin and $L_i(0)-L_i(-1)=1$, so the assertion still holds. Finally,
\begin{align*}
A_i(t)&=\sum_{r\ge0}L_i(r)t^r-\sum_{r\ge1}L_i(r-1)t^r\\
 &=\Ehr_{\tau_i}(t)-t\sum_{s\ge0}L_i(s)t^s
 =(1-t)\Ehr_{\tau_i}(t),
\end{align*}
where the second sum is reindexed by $s=r-1$. The equalities hold in the ring of formal power series.
\end{proof}

\begin{proof}[Proof of Theorem~\ref{main}]
First suppose that both ground sets are nonempty. The lattice-point bijection and the identity for the scales imply that, for integers $m\ge0$,
\begin{equation}\label{countconvolution}
L_{\tau_1\oplus\tau_2}(m)
 =\sum_{\substack{r,s\ge0\\r+s\le m}}a_1(r)a_2(s).
\end{equation}
Indeed, each lattice point in the source has a unique pair of integer scales $(r,s)$. Its total scale is at most $m$ if and only if $r+s\le m$. At fixed scales, the two coordinate blocks can be chosen independently, giving $a_1(r)a_2(s)$ choices. The layers are disjoint, so no lattice point is counted twice.

Multiply~\eqref{countconvolution} by $t^m$ and sum over $m\ge0$. For any fixed coefficient of $t^m$, only finitely many pairs $r,s$ satisfy $r+s\le m$, so we may interchange the sums in the ring of formal power series. This yields
\begin{align*}
\Ehr_{\tau_1\oplus\tau_2}(t)
 &=\sum_{r,s\ge0}a_1(r)a_2(s)\sum_{m\ge r+s}t^m\\
 &=\frac{A_1(t)A_2(t)}{1-t}
 =(1-t)\Ehr_{\tau_1}(t)\Ehr_{\tau_2}(t).
\end{align*}
Write $d_1=|A|$ and $d_2=|B|$. By Lemma~\ref{interior}, the dimensions of the two factors and the target are $d_1,d_2,d_1+d_2$, respectively. Substituting~\eqref{ehrdefinition}, the right-hand side becomes
\[
\frac{(1-t)h_{\tau_1}^*(t)h_{\tau_2}^*(t)}
 {(1-t)^{d_1+1}(1-t)^{d_2+1}}
 =\frac{h_{\tau_1}^*(t)h_{\tau_2}^*(t)}{(1-t)^{d_1+d_2+1}}.
\]
Multiplying both sides by the common denominator proves~\eqref{hproduct}.

If one ground set is empty, the ordinal sum is the other preorder and the $h^*$-polynomial of the zero-dimensional point is one, so the identity holds. If both ground sets are empty, it reads $1=1\cdot1$. Thus inserting or deleting empty factors does not affect any conclusion.
\end{proof}

This counting by layers is Braun's free-sum formula~\cite[Theorem~1]{BRAUN} in the present setting. The functional identity~\eqref{addidentity} transfers the count from the free sum to the ordinal sum, while Proposition~\ref{gauge} gives the integer-valued property directly.

The algebraic interpretation of Bruns~\cite[Section~1]{BRUNS} can be phrased using the following object:
\[
 S_P=\{(z,m):m\in\mathbb Z_{\ge0},\ z\in mP\cap\mathbb Z^d\},
 \qquad P\subset\mathbb R^d.
\]
Here $P$ is a lattice polytope. Convexity ensures that $S_P$ is closed under vector addition, making it a monoid, called the Ehrhart monoid of $P$. The last coordinate gives the degree, and counting elements by degree gives the Ehrhart series. The algebraic construction of the free sum identifies the elements $(0,1)$ in the two factors; the resulting degree-one binomial relation accounts for the factor $1-t$.

Our shear does not preserve addition. For example, when both factors are singletons, $\Phi(0,1)=(1,1)$ and $\Phi(0,-1)=(0,-1)$, whose sum differs from $\Phi(0,0)$. Thus the shear gives a lattice-point bijection preserving dilation scales, but cannot directly be viewed as an isomorphism of these monoids. This also explains the role of counting by layers here.

\subsection{Several factors and normalized volume}\label{sec:consequences}
Iterating the product formula along any finite ordinal sum also gives a product formula for normalized volume. Since normalized volume is the value of the $h^*$-polynomial at one, the two formulas have the same factorization.

For a full-dimensional lattice polytope $P\subset\mathbb R^d$, write $L_P(m)=\#(mP\cap\mathbb Z^d)$ for $m\in\mathbb Z_{\ge0}$, and put
$\Vol_{\mathbb Z}(P)=d!\operatorname{vol}_d(P)$, where the unit cube has Euclidean volume one. The normalized volume of a zero-dimensional point is defined to be one.

\begin{corollary}\label{multiple}
Let $k\ge1$ be an integer. If $\tau_1,\ldots,\tau_k$ are finite preorders on pairwise disjoint ground sets, then
\begin{align*}
 h_{\tau_1\oplus\cdots\oplus\tau_k}^*(t)
   &=\prod_{j=1}^k h_{\tau_j}^*(t),\\
 \Vol_{\mathbb Z}(\Rpoly{\tau_1\oplus\cdots\oplus\tau_k})
   &=\prod_{j=1}^k\Vol_{\mathbb Z}(\Rpoly{\tau_j}).
\end{align*}
\end{corollary}
\begin{proof}
An ordinal sum retains the internal relations of its factors and places each element of an earlier ground set below every element of a later ground set. This description is independent of parentheses, so ordinal sum is associative. For $k=1$ the formula is an identity. If it holds for the first $k-1$ factors, applying Theorem~\ref{main} to their ordinal sum and the $k$th factor proves it for $k$ factors. Empty factors may be deleted by the preceding convention.

To see why evaluation at $t=1$ gives volume, write $h^*(P,t)=\sum_{j=0}^d h_j^*t^j$. Extracting coefficients using
$(1-t)^{-d-1}=\sum_{m\ge0}\binom{m+d}{d}t^m$
yields, for $m\ge d$,
\[
L_P(m)=\sum_{j=0}^d h_j^*\binom{m+d-j}{d}.
\]
Each binomial coefficient is a polynomial in $m$ of degree $d$ with leading coefficient $1/d!$. The leading coefficient of $L_P$ is $\operatorname{vol}_d(P)$ by Ehrhart's theorem; see~\cite[Section~3.6]{BR}. Comparing leading coefficients gives
$d!\operatorname{vol}_d(P)=\sum_j h_j^*=h^*(P,1)$.
Evaluating the product formula for $h^*$ at $t=1$ gives the normalized-volume formula.
\end{proof}

\subsection{The smallest example without affine equivalence}
Even when each factor has only one element, the ordinal-sum polytope is not affinely equivalent to the free sum. This example shows both why a piecewise shear is needed and how it preserves the entire Ehrhart count.

When both factors have one element, $\Rpoly{\tau_i}=[-1,1]$. Its $m$th integral dilate contains $2m+1$ lattice points. The identity
$\sum_{m\ge0}t^m=(1-t)^{-1}$ and formal differentiation give
\[
\sum_{m\ge0}(2m+1)t^m
 =\frac{2t}{(1-t)^2}+\frac1{1-t}
 =\frac{1+t}{(1-t)^2}.
\]
Thus each factor has $h^*$-polynomial $1+t$. The source free sum is the diamond with vertices $(\pm1,0),(0,\pm1)$. The nonempty order ideals of the two-element chain are $\{1\}$ and $\{1,2\}$, so the target is
\[
\conv\{(-1,0),(0,-1),(1,0),(1,1)\}.
\]
The shear $\Phi(x,y)=(x+\max(0,y),y)$ is the identity on the lower half-plane and is $(x,y)\mapsto(x+y,y)$ on the upper half-plane. The two pieces of the source are the triangles generated by $0,(-1,0),(1,0),(0,-1)$ and by $0,(-1,0),(1,0),(0,1)$, respectively. Their images under the linear pieces are the lower and upper triangles of the target. They meet along $y=0$ and together form the target quadrilateral.

Let $P$ denote the target quadrilateral. The main theorem gives its $h^*$-polynomial as $(1+t)^2$ and its normalized area as four. Further coefficient extraction yields
\[
L_P(m)=\binom{m+2}{2}+2\binom{m+1}{2}+\binom m2=2m^2+2m+1
\]
for every $m\ge0$, with a binomial coefficient taken to be zero when its nonnegative upper index is less than two. The source is a parallelogram, whereas the target is not: the vertices of the target in boundary order are $(-1,0),(0,-1),(1,0),(1,1)$, and the sums of opposite vertices are $(0,0)$ and $(1,0)$. Affine maps preserve parallelograms, so the two polytopes are not affinely equivalent. Nevertheless, the piecewise shear gives a lattice-point bijection at every scale.

\section{Conclusions and further questions}\label{sec:discussion}
We have proved the ordinal-sum conjecture of Athanasiadis and Chapoton and explained its geometry: a single explicit piecewise unimodular shear relates the polar preorder polytope of an ordinal sum to the corresponding free sum. The least nonnegative order-reversing majorant gives a formula for the Minkowski functional, whose additivity turns the geometric correspondence into a convolution over lattice-point scales. Both the $h^*$-factorization and the normalized-volume product therefore hold for arbitrary finite preorders.

The result does not determine the $h^*$-polynomial of a general preorder. It suggests two related directions: whether reversing the order relation also admits a geometric explanation preserving lattice points, and whether the product formula comes from a direct combinatorial decomposition.

\subsection{Opposite preorders and geometric bijections}
First consider reversing every order relation. Write $\tau^{\mathrm{op}}$ for the opposite preorder, so that $i\le_{\tau^{\mathrm{op}}}j$ if and only if $j\le_\tau i$. This is the object denoted by $\tau^*$ in the original paper; we use the usual superscript $\mathrm{op}$ to distinguish it from $h^*$. Another conjecture of Athanasiadis and Chapoton is the following~\cite[Conjecture~6.2]{AC}.
\begin{conjecture}\label{dualconjecture}
For every finite preorder $\tau$,
\[
 h_\tau^*(t)=h_{\tau^{\mathrm{op}}}^*(t).
\]
\end{conjecture}
This conjecture compares the polytopes of a preorder and its opposite, rather than the coefficients of a single polynomial. Indeed, put $d=|E|$. Reflexivity and Hibi's theorem~\cite[Theorem~2.1]{HIBI} give
\[
 \Ehr_\tau(t^{-1})=(-1)^{d+1}t\Ehr_\tau(t).
\]
Substituting~\eqref{ehrdefinition} and simplifying with $(1-t^{-1})^{d+1}=(-1)^{d+1}t^{-d-1}(1-t)^{d+1}$ yields
\[
 t^{|E|}h_\tau^*(t^{-1})=h_\tau^*(t),
\]
This identity itself does not involve $\tau^{\mathrm{op}}$.

The main theorem reduces Conjecture~\ref{dualconjecture} to preorders that cannot be written as the ordinal sum of two nonempty preorders. Indeed,
\[
 \tau=\tau_1\oplus\cdots\oplus\tau_k
 \quad\Longrightarrow\quad
 \tau^{\mathrm{op}}=\tau_k^{\mathrm{op}}\oplus\cdots\oplus\tau_1^{\mathrm{op}}.
\]
If the conjecture holds for each factor, the two $h^*$-polynomials are products of the same factors in reverse order. Each nontrivial decomposition decreases the size of the ground set, so repeated decomposition terminates. In particular, when the quotient poset is a total order, each factor consists of a single equivalence class and is equal to its opposite. The conjecture therefore holds in this case.

Our construction suggests the following stronger question.
\begin{problem}\label{dualgeometry}
Let $\tau$ be a preorder on a finite set $E$. Does there exist a positively homogeneous piecewise unimodular homeomorphism
$T_\tau:\mathbb R^E\to\mathbb R^E$ such that
\[
 T_\tau(\Rpoly{\tau})=\Rpoly{\tau^{\mathrm{op}}}?
\]
If so, can $T_\tau$ be chosen as a composition of unimodular linear maps and finitely many maps of the form~\eqref{mutationmap}, with every intermediate polytope image remaining convex?
\end{problem}
An affirmative answer to the first question would give lattice-point bijections at all scales and hence prove Conjecture~\ref{dualconjecture}. The second question further asks that the correspondence arise from combinatorial mutation. Our single mutation connects the ordinal-sum polytope to the free sum; opposite preorders require new transformations on factors that are indecomposable under ordinal sum. For a general preorder, reversing the relation turns upper sets into lower sets, and there is no longer a common maximum that can translate all coordinates at once. The difficulty is to find local transformations that handle these different sets simultaneously while preserving convexity, rather than merely matching lattice points separately at each scale.

\subsection{A combinatorial interpretation of the product formula}
A second direction is to interpret the coefficients of $h_\tau^*(t)$ directly. The $h^*$-coefficients of a lattice polytope are always nonnegative integers~\cite[Section~3.5]{BR}, but this general fact does not identify objects naturally defined from the preorder that they count. Multiplicativity under ordinal sum places a concrete requirement on such an interpretation.

\begin{problem}\label{combinatorial}
Can one construct directly from the order relation of $\tau$ a finite set of combinatorial objects $\mathcal W(\tau)$ and a statistic
$\operatorname{stat}_\tau:\mathcal W(\tau)\to\{0,\ldots,|E|\}$ such that
\[
 h_\tau^*(t)=\sum_{w\in\mathcal W(\tau)}t^{\operatorname{stat}_\tau(w)},
\]
and such that ordinal sum corresponds to a natural bijection
\[
 \mathcal W(\tau_1\oplus\tau_2)
       \simeq\mathcal W(\tau_1)\times\mathcal W(\tau_2),
\]
under which the statistic is the sum of the two statistics on the right? The objects and their statistics should be defined directly from the order relation and be compatible with relabelling the ground set: each preorder isomorphism should induce a statistic-preserving bijection of objects, and these bijections should respect composition of isomorphisms.
\end{problem}
One concrete approach would be to use chains of order ideals carrying integer labels. The permitted chain lengths, label ranges, and relations between consecutive terms would be specified explicitly by $\tau$, and the statistic would be the sum of the label contributions at each step. The aim is then not merely to find a finite set with a prescribed generating function: ordinal-sum decomposition should correspond to splitting a chain, and relabelling the ground set should leave the statistic unchanged. The admissibility conditions must not depend on previously computed $h^*$-coefficients. Whether such a model exists remains open.

The appendix provides a case that may guide a general construction. If the equivalence classes are totally ordered and have sizes $d_1,\ldots,d_k$, one may take
\[
 \mathcal W(\tau)=\prod_{j=1}^k\{0,1,\ldots,d_j\},
 \qquad \operatorname{stat}_\tau(a_1,\ldots,a_k)=a_1+\cdots+a_k.
\]
In this model, multiplication comes from the Cartesian product, coefficient symmetry comes from
$(a_j)_j\mapsto(d_j-a_j)_j$, and reversing the preorder corresponds to reversing the order of the coordinate blocks. A model for general preorders compatible with these operations would turn the present geometric proof into a correspondence between combinatorial objects. What is missing is a uniform construction for factors indecomposable under ordinal sum: our functional formula determines their lattice-point scales but does not directly provide a finite model with a combinatorial statistic.

\clearpage
\appendix
\section*{Appendix: equivalence-class sizes and lattice-point enumeration}
The order ideals of a preorder are unions of equivalence classes, but the lattice-point counts of its polar polytope are not determined by the quotient poset alone: the sizes of the equivalence classes also enter the $h^*$-polynomial. Two successive calculations illustrate this dependence and give a family to which the main theorem applies directly.

Appendix~\ref{app:single} starts with a single equivalence class. Its polytope is a simplex with $h^*$-polynomial $1+t+\cdots+t^d$; lattice points in a half-open fundamental parallelepiped explain why each coefficient is one. Appendix~\ref{app:equivalence} then joins these preorders by ordinal sum, giving an explicit product when the equivalence classes are totally ordered. This shows why their sizes cannot be ignored and provides the simplest examples for the combinatorial-model question in Section~\ref{sec:discussion}.

\section{The simplex of a single equivalence class}\label{app:single}

Suppose that $\tau$ has $d\ge1$ elements, all equivalent. The only nonempty order ideal is the whole ground set, so $\Rpoly{\tau}$ is the simplex in Lemma~\ref{interior}. We prove that its $h^*$-polynomial is $1+t+\cdots+t^d$. The coefficients, all equal to one, can be seen directly from the lattice points in the cone over the simplex.

Label the ground set by $1,\ldots,d$ and put
\[
 v_0=(1,\ldots,1),\qquad v_i=-e_i\quad(1\le i\le d),
 \qquad S_d=\conv(v_0,\ldots,v_d).
\]
\begin{proposition}\label{equivalenceclass}
We have
\[
 \Ehr_{S_d}(t)=\frac{1+t+\cdots+t^d}{(1-t)^{d+1}},\qquad
 h^*(S_d,t)=1+t+\cdots+t^d.
\]
\end{proposition}
\begin{proof}
In $\mathbb R^{d+1}$, put $w_i=(v_i,1)$ and call the last coordinate the height. Since $v_0,\ldots,v_d$ are affinely independent, $w_0,\ldots,w_d$ are linearly independent. The cone they generate,
\[
 C=\left\{\sum_{i=0}^d\lambda_iw_i:\lambda_i\ge0\right\}
\]
has section $\{(z,m):z\in mS_d\}$ at integer height $m$. Indeed, height equals $\sum_i\lambda_i$, so division by $m$ gives a convex combination when $m>0$; at height zero, all coefficients vanish. Counting lattice points of $C$ by height therefore gives $\Ehr_{S_d}(t)$.

For each lattice point of the cone, write uniquely $\lambda_i=n_i+a_i$, where $n_i\in\mathbb Z_{\ge0}$ and $0\le a_i<1$. Subtracting the integer vector $\sum_i n_iw_i$ leaves a lattice point in the half-open parallelepiped
\[
 \Pi=\left\{\sum_{i=0}^d a_iw_i:0\le a_i<1\right\}
\]
Conversely, adding any $\sum_i n_iw_i$ to a lattice point of $\Pi$ gives a lattice point of $C$. This representation is unique, so
\[
 \Ehr_{S_d}(t)=\frac{\sum_{q\in\Pi\cap\mathbb Z^{d+1}}t^{q_{d+1}}}
 {(1-t)^{d+1}}.
\]
Each factor $(1-t)^{-1}$ in the denominator records one of the $n_i$, since every $w_i$ has height one.

We now determine the lattice points of $\Pi$. The $i$th spatial coordinate of $\sum_i a_iw_i$ is $a_0-a_i$. It is an integer strictly between $-1$ and one, so it is zero; hence $a_i=a_0$. The last coordinate is therefore $(d+1)a_0$, which is an integer if and only if
\[
 a_0=\frac{k}{d+1},\qquad k=0,1,\ldots,d.
\]
All these values are allowed, and the resulting lattice points are exactly $(0,k)$, with one at each height. Substitution into the series formula proves the assertion.
\end{proof}

\section{The product formula for totally ordered equivalence classes}\label{app:equivalence}
If the quotient poset is a total order, the original preorder is an ordinal sum of preorders each consisting of a single equivalence class. The simplex calculation and the main theorem therefore give every $h^*$-coefficient and show explicitly how it depends on the sizes of the equivalence classes.

Let the equivalence classes be $E_1<\cdots<E_k$, with $d_j=|E_j|$. The preorder on each $E_j$ has a single equivalence class, so Corollary~\ref{multiple} and Proposition~\ref{equivalenceclass} give
\[
 h_\tau^*(t)=\prod_{j=1}^k(1+t+\cdots+t^{d_j}),\qquad
 \Vol_{\mathbb Z}(\Rpoly{\tau})=\prod_{j=1}^k(d_j+1).
\]
In particular, the $h^*$-polynomial of a total order on $d$ elements is $(1+t)^d$, whereas that of $d$ equivalent elements is $1+t+\cdots+t^d$. In the latter case, the quotient poset has only one element, but the polytope still has dimension $d$. This explains why we retain all coordinates of the original ground set. Moreover, the product depends only on the multiset of equivalence-class sizes: reordering these totally ordered classes leaves the $h^*$-polynomial unchanged.

\end{document}